\documentclass[draft]{article} 
\usepackage{amsmath,amssymb,amsthm}

\newtheorem{thm}{Theorem}[section]
\newtheorem{lem}[thm]{Lemma}
\newtheorem{prop}[thm]{Proposition}
\theoremstyle{remark}
\newtheorem{rem}[thm]{Remark}
\theoremstyle{definition}

\begin{document}

\title{Certain identities on derivatives 
of radial homogeneous 
and logarithmic functions}

\author{Kei Morii 
\thanks{Department of Mathematics, Kyoto University, Kyoto, 606-8502, Japan. \protect \\
\textit{E-mail address:} \texttt{k-morii@math.kyoto-u.ac.jp}}, 
\and
        Tokushi Sato 
\thanks{Mathematical Institute, Tohoku University, Sendai, 980-8578, Japan. \protect \\
\textit{E-mail address:} \texttt{tokushi@math.tohoku.ac.jp}}, 
\and
        Yoshihiro Sawano 
\thanks{Department of Mathematics, Kyoto University, Kyoto, 606-8502, Japan. \protect \\
\textit{E-mail address:} \texttt{yosihiro@math.kyoto-u.ac.jp}}, 
}
\allowdisplaybreaks

\maketitle

\begin{abstract}
Let $k$ be a natural number and $s$ be real. 
In the 1-dimensional case, 
the $k$-th order derivatives 
of the functions $\lvert x\rvert^s$ 
and $\log \lvert x\rvert$ 
are multiples of $\lvert x\rvert^{s-k}$ 
and $\lvert x\rvert^{-k}$, 
respectively. 
In the present paper,
we generalize this fact to higher dimensions
by introducing a suitable norm
of the derivatives, 
and give the exact values 
of the multiples. 
\end{abstract}

\section{Introduction}
\label{sect:intro}
\par
In the present paper, 
we show two identities 
for derivatives 
of radial homogeneous functions 
and a radial logarithmic function. 
A logarithm $\log r$ always stands for 
the natural logarithm $\log_e r$. 
Let $k\in\mathbb{N}=\{1,2,\ldots\}$ and $s\in\mathbb{R}$. 
In the 1-dimensional case, 
we readily have that 
the functions $(d/dx)^k[\lvert x\rvert^s]$, 
$(d/dx)^k[\log\lvert x\rvert]$ 
are homogeneous of degree $s-k$, 
$-k$, 
respectively. 
Precisely we have 
\begin{gather}
\label{eq:1dim_pwr}
\lvert x\rvert^{k-s}
\left\lvert
\left(
\frac{d}
     {dx}
\right)^k
[
\lvert x\rvert^s
]
\right\rvert
=
\left\lvert
(s)_k
\right\rvert
,\;
\lvert x\rvert^k
\left\lvert
\left(
\frac{d}
     {dx}
\right)^k
\left[
\log
\lvert x\rvert
\right]
\right\rvert
=
(k-1)!
\\\notag
\text{for\;}
x\in\mathbb{R}\setminus\{0\}.
\end{gather}
Here we use the Pochhammer symbol 
for the falling factorial (lower factorial);
$$
(\nu)_k
=
\begin{cases}
\displaystyle
\prod_{j=0}^{k-1}
(\nu-j)
&
\text{for\;}
\nu\in\mathbb{R},\,
k\in\mathbb{N},
\\
1
&
\text{for\;}
\nu\in\mathbb{R},\,
k=0.
\end{cases}
$$
We denote the space dimension 
by $N\in\mathbb{N}$. 
Let $\nabla^k$ be 
a partial differential operator on $\mathbb{R}^N$ 
which contains only $k$-th order derivatives. 
Then the functions 
$\nabla^k[\lvert x\rvert^s]$, 
$\nabla^k[\log\lvert x\rvert]$ 
for $x\in \mathbb{R}^N\setminus\{0\}$ 
are also homogeneous of degree $s-k$, 
$-k$, 
respectively. 
However, 
it is not trivial 
whether the functions
\begin{equation}
\label{eq:cnst_fn}
\lvert x\rvert^{k-s}
\lvert \nabla^k[\lvert x\rvert^s]\rvert,
\, 
\lvert x\rvert^k
\lvert \nabla^k[\log\lvert x\rvert]\rvert
\end{equation}
are constants or not. 
It deeply depends on 
the definition of the norm $\lvert \nabla^k u(x)\rvert$ 
of the vector $\nabla^k u(x)$ 
for a smooth function $u$ 
defined on a domain in $\mathbb{R}^N$. 
See Remark \ref{rem:bad_example} below 
for a counterexample.
\par
In the present paper, 
we shall define an appropriate norm 
of the vector $\nabla^k u(x)$ 
to solve this problem affirmatively, 
and specify the constants in (\ref{eq:cnst_fn}). 
\par
In what follows, 
we specify the dimension $N$ as a sub- or super-script 
and denote by 
$\lvert\,\cdot\,\rvert_N$ the Euclidean norm on $\mathbb{R}^N$; 
$$ 
\lvert x\rvert_N
=
(x_1^2+x_2^2+\cdots+x_N^2)^{1/2}
\text{\;\;for\;}
x=(x_1,x_2,\ldots,x_N)\in\mathbb{R}^N.
$$
Let us write $I_N=\{1,2,\ldots,N\}$ for short. 
For a $k$-tuple of indices $i=(i_1,i_2,\ldots,i_k)\in I_N^k$, 
define the $k$-th order partial differential operator $D_i$ 
as
$$
D_i
=
D_{i_1} 
D_{i_2} 
\cdots
D_{i_k} 
=
\frac{\partial}
     {\partial x_{i_1}}
\;\!
\frac{\partial}
     {\partial x_{i_2}}
\cdots
\frac{\partial}
     {\partial x_{i_k}}. 
$$
For a smooth real-valued function $u$ 
on a domain $\Omega$ in $\mathbb{R}^N$, 
define the vector 
$$
\nabla_{\!N}^k u(x)
=
(D_i u(x))_{i\in I_N^k}
\text{\;\;for\;}
x\in\Omega
$$
and its norm as 
\begin{align*}
\lvert 
\nabla_{\!N}^k u(x)
\rvert_{N^k}
&=
\left(
\sum_{i\in I_N^k}
  (D_i u(x))^2
\right)^{1/2}
\\
&=
\left(
\sum_{i_1=1}^N
\sum_{i_2=1}^N
\cdots
\sum_{i_k=1}^N
  \left(
    \frac{\partial}
         {\partial x_{i_1}}
    \frac{\partial}
         {\partial x_{i_2}}
\cdots
    \frac{\partial}
         {\partial x_{i_k}}
    u(x)
  \right)^2
\right)^{1/2}
\text{\;\;for\;}
x\in\Omega;
\end{align*}
we make the agreement $\nabla_{\!N}^0 u(x)=u(x)$ 
and then $\lvert\nabla_{\!N}^0 u(x)\rvert_1=\lvert u(x)\rvert$. 
When $k=1$, 
$\nabla_N^1 u(x)$ coincides 
with the gradient vector of $u(x)$, 
and $\lvert\nabla_N^1 u(x)\rvert_N$ is its Euclidean norm. 
When $k=2$, 
$\nabla_N^2 u(x)$ can be identified 
with the Hessian matrix of $u(x)$, 
and $\lvert\nabla_N^2 u(x)\rvert_{N^2}$ is its Frobenius norm. 
Then we have the following results. 
Let $\mathbb{Z}_+=\{0\}\cup\mathbb{N}=\{0,1,2,\ldots\}$. 
%
\par
\begin{thm}
\label{thm:derivative_log_abs}
Let $N\in\mathbb{N}$. 
\begin{itemize}
\item[{\rm (i)}]
For any 
$k\in\mathbb{Z}_+$ and $s\in\mathbb{R}$, 
there exists a constant $\gamma_N^{s,k}\ge 0$ 
such that 
$$
(
\lvert x\rvert_N^{k-s}
\lvert
\nabla_{\!N}^k
[
\lvert x\rvert_N^s
]
\rvert_{N^k}
)^2
=
\gamma_N^{s,k}
\text{\;\;for\/\;}
x\in\mathbb{R}^N\setminus\{0\}.
$$
\item[{\rm (ii)}]
For any $k\in\mathbb{N}$, 
there exists a constant $\ell_N^k>0$ 
such that 
$$
(
\lvert x\rvert_N^k
\lvert
\nabla_{\!N}^k
[
\log
\lvert x\rvert_N
]
\rvert_{N^k}
)^2
=
\ell_N^k
\text{\;\;for\/\;}
x\in\mathbb{R}^N\setminus\{0\}.
$$
\end{itemize}
\end{thm}
\par
It follows from (\ref{eq:1dim_pwr}) that 
for any $k\in\mathbb{N}$ and $s\in\mathbb{R}$, 
\begin{equation}
\label{eq:1-dim_result}
\gamma_1^{s,k}
=
((s)_k)^2,
\;
\ell_1^k
=
((k-1)!)^2.
\end{equation}
We can determine explicitly 
the constants $\gamma_N^{s,k}$ and $\ell_N^k$ 
given in Theorem \ref{thm:derivative_log_abs} 
for a general dimension $N$ 
as follows. 
Before we go into the detail, 
we provide some notation. 
Let 
$$
\lfloor \nu\rfloor
=
\max
\{k\in\mathbb{Z};\,k\le\nu\}
,\;
\lceil \nu\rceil
=
\min
\{k\in\mathbb{Z};\,k\ge\nu\}
\text{\;\;for\;}
\nu\in\mathbb{R}.
$$
Define the binomial coefficient 
$$
{\nu\choose k}
=
\frac{(\nu)_k}
     {k!}
\text{\;\;for\;}
\nu\in\mathbb{R}
,\,
k\in\mathbb{Z}_+.
$$
%
\par
The following theorem provides 
the explicit values of the constants $\gamma_N^{s,k}$ and $\ell_N^k$.
\par
\begin{thm}
\label{thm:ind}
Let $N\in\mathbb{N}$. 
\begin{itemize}
\item[{\rm (i)}]
For any $k\in\mathbb{Z}_+$ 
and $s\in\mathbb{R}$,  
it holds
\begin{align*}
\gamma_N^{s,k}
&=
k!
\sum_{l=0}^{\lfloor k/2\rfloor}
(k-2l)!\,
l!\,
\left(
\frac{N-3}
     {2}
+l
\right)_l
\\*
&\qquad{}\times
\left(
\sum_{n=\lceil k/2\rceil}^{k-l}
2^{2n-k+l}
{s/2 \choose n}
{n \choose k-n}
{k-n\choose l}
\right)^2.
\end{align*}
\item[{\rm (ii)}]
For any $k\in\mathbb{N}$, 
it holds
\begin{align*}
\ell_N^k
&=
k!
\sum_{l=0}^{\lfloor k/2\rfloor}
(k-2l)!\,
l!\,
\left(
\frac{N-3}
     {2}
+l
\right)_l
\\*
&\qquad{}\times
\left(
\sum_{n=\lceil k/2\rceil}^{k-l}
2^{2n-k+l}
\frac{(-1)^n}
     {2n}
{n \choose k-n}
{k-n\choose l}
\right)^2.
\end{align*}
\end{itemize}
\end{thm}
\par
We also obtain the following result 
as a special case of Theorem \ref{thm:ind}. 
\par
\begin{thm}
\label{thm:2-dim}
\begin{itemize}
\item[{\rm (i)}]
For any $N\in\mathbb{N}$ and $k\in\mathbb{Z}_+$, 
it holds
$$
\gamma_N^{-(N-2),k}
=
2^k
\left(
\frac N2+k-2
\right)_k
(N+k-3)_k
.
$$ 
\item[{\rm (ii)}]
For any $k\in\mathbb{N}$, 
it holds
$$
\ell_2^k
=
2^{k-1}((k-1)!)^2.
$$ 
\end{itemize}
\end{thm}
\par
\begin{rem}
For small $k$, 
we have calculated
the concrete values of $\gamma_N^{s,k}$
and $\ell_N^k$;
\begin{align*}
&
\gamma_N^{s,1}
=
s^2
,\;
\gamma_N^{s,2}
=
s^2 (s^2-2s+N)
,\;
\gamma_N^{s,3}
=
s^2 (s-2)^2 (s^2-2s+3N-2)
,
\\
&
\gamma_N^{s,4}
=
s^2 (s-2)^2
(s^4-8 s^3+(16+6N) s^2+(12-36N)s+3 N^2+54 N-48)
,
\\
&
\ell_N^1=1
,\;
\ell_N^2=N
,\;
\ell_N^3
=
4(3N-2)
,\;
\ell_N^4
=
12(N^2+18N-16)
,
\\
&
\ell_N^5
=
192(5N^2+30N-32)
,\;
\ell_N^6
=
960(N^3+78N^2+224N-288), 
\\
&
\ell_N^7
=
34560(7N^3+196N^2+308N-496), 
\\
&
\ell_N^8
=
241920(N^4+204N^3+3052N^2+2736N-5888). 
\end{align*}
\end{rem}
\par
%
%
%
\par
\par
As we mentioned before, 
it is essential to define the norm 
$\lvert\nabla^k u(x)\rvert$ 
appropriately. 
\par
\begin{rem}
\label{rem:bad_example}
One may also adopt some other 
plausible definition instead of $\lvert\nabla_{\!N}^k u(x)\rvert_{N^k}$ defined before. 
For instance, 
let us define 
\begin{align*}
\lvert\tilde\nabla_{\!N}^k u(x)\rvert_
{{N+k-1\choose k}}
&=
\left(
\sum_{\alpha_1+\alpha_2+\cdots+\alpha_N=k}
(
D_1^{\alpha_1}
D_2^{\alpha_2}
\cdots
D_N^{\alpha_N}
 u(x))^2
\right)^{1/2} 
\\
&=
\left(
\sum_{1\le i_1\le i_2\le\cdots\le i_k\le N}
(
D_{i_1}
D_{i_2}
\cdots
D_{i_k}
    u(x)
)^2
\right)^{1/2},
\end{align*}
which gives a norm of $\nabla^k u(x)$. 
Putting $k=2$, 
we see that both the functions
\begin{gather*}
(
\lvert x\rvert_N^{2-s}
\lvert
\tilde\nabla_{\!N}^2
[
\lvert x\rvert_N^s
]
\rvert_{N(N+1)/2}
)^2
=
s^2
\left(
  N+2s-4+
(s-2)^2
\sum_{1\le i_1\le i_2\le N}
\frac{x_{i_1}^2 x_{i_2}^2}
     {\lvert x\rvert_N^4}
\right),
\\
(
\lvert x\rvert_N^2
\lvert
\tilde\nabla_{\!N}^2
[
\log
\lvert x\rvert_N
]
\rvert_{N(N+1)/2}
)^2
=
  N-4+
4
\sum_{1\le i_1\le i_2\le N}
\frac{x_{i_1}^2 x_{i_2}^2}
     {\lvert x\rvert_N^4}
\end{gather*}
are not constants on $\mathbb{R}^N\setminus\{0\}$ 
unless $N=1$ or $s\in\{0,2\}$. 
To illustrate how they are different clearly, 
note that
\begin{align*}
\lvert\nabla_{\!N}^k u(x)\rvert_{N^k}
&=
\left(
\sum_{\alpha_1+\alpha_2+\cdots+\alpha_N=k}
\frac{k!}
     {\alpha_1!\,\alpha_2!\cdots\alpha_N!}
(
D_1^{\alpha_1}
D_2^{\alpha_2}
\cdots
D_N^{\alpha_N}
 u(x))^2
\right)^{1/2}
\\
&=
\left(
\sum_{1\le i_1\le i_2\le\cdots\le i_k\le N}
\frac{k!}
     {\displaystyle\prod_{l=1}^N \sharp\{n;\,i_n=l\}!}
(
D_{i_1}
D_{i_2}
\cdots
D_{i_k}
    u(x)
)^2
\right)^{1/2},
\end{align*}
where $\sharp S$ denote the cardinality 
of a finite set $S$. 
\end{rem}
\par
The present work is originated 
in our desire to investigate 
Br\'ezis-Gallou\"et-Wainger type inequalities. 
The authors together with Wadade \cite{MSSW1}, 
\cite{MSW1} and \cite{MSW2} 
investigated the sharp constants 
of 
such inequalities 
in the first order critical Sobolev space 
$W_0^{1,N}(\Omega)$ 
on a bounded domain $\Omega$ in $\mathbb{R}^N$ 
with $N\in\mathbb{N}\setminus\{1\}$. 
In their forthcoming paper \cite{MSS2}, 
they shall give a lower bound 
in terms of $\ell_N^k$ 
for the sharp constants 
of such inequalities 
in the $k$-th order critical Sobolev space 
$W_0^{k,N/k}(\Omega)$ 
by calculating the exact values 
of homogeneous Sobolev norms 
of the radial logarithmic function 
on annuli. 
To explain more concretely, 
we can give a sufficient condition 
for $\lambda_1>0$ and $\lambda_2\in\mathbb{R}$ 
that the inequality 
\begin{gather*}
\begin{aligned}
\label{eq:DL_X}
\lVert u\rVert_{L^\infty(\Omega)}
^{N/(N-k)}
&\le
\lambda_1
\log
(
  1+
\lVert u\rVert_{A^{s,N/(s-\alpha),q}(\Omega)}
)
\\*
&\quad{}+
\lambda_2
\log
(
  1+
\log
(
  1+
\lVert u\rVert_{A^{s,N/(s-\alpha),q}(\Omega)}
)
)
+C
\end{aligned}
\\
\text{for\;}
u\in W^{k,N/k}_0(\Omega)\cap A^{s,N/(s-\alpha),q}(\Omega)
\text{\;\;with\;}
\Vert\nabla^k  u\Vert_{L^{N/k}(\Omega)}=1
\end{gather*}
fails for any constant $C$ independent of $u$, 
where $k\in\{1,2,\ldots,N-1\}$, 
$0<\alpha\le s<\infty$, 
$0<q<\infty$  
and we denote by $A^{s,p,q}$ 
either the Besov space $B^{s,p,q}$ 
or the Triebel-Lizorkin space $F^{s,p,q}$. 
The results in \cite{Brezis2} and \cite{Brezis1} 
obtained by Br\'ezis, 
Gallou\"et and Wainger imply that  
this inequality holds 
for sufficiently large $\lambda_1$ 
and arbitrary $\lambda_2$ 
with a suitable constant $C$ 
provided that 
$A^{s,N/(s-\alpha),q}(\Omega)$ is replaced by 
the Sobolev space 
(or the potential space) 
$H^{s,N/(s-\alpha)}(\Omega)$. 
Since $H^{s,N/(s-\alpha)}(\Omega)=F^{s,N/(s-\alpha),2}(\Omega)$, 
the same assertion holds in the cases 
$A^{s,N/(s-\alpha),q}(\Omega)=F^{s,N/(s-\alpha),q}(\Omega)$ with $0<q\le 2$ 
and 
$A^{s,N/(s-\alpha),q}(\Omega)=B^{s,N/(s-\alpha),q}(\Omega)$ with $0<q\le \min\{N/(s-\alpha),2\}$ 
by virtue of the embedding theorems 
of Besov and Triebel-Lizorkin spaces.
\par
We now describe how we organized the present paper; 
Sections \ref{sect:prf_1}, 
\ref{sect:prf_2} and \ref{sect:prf_3} 
are devoted to proving 
Theorems \ref{thm:derivative_log_abs}, 
\ref{thm:ind} and \ref{thm:2-dim}, 
respectively. 
\par
\section{Proof of Theorem \ref{thm:derivative_log_abs}}
\label{sect:prf_1}
\par
In this section, 
we shall prove Theorem \ref{thm:derivative_log_abs}. 
The following two propositions are easy to prove. 
\par
\begin{prop}
\label{prop:hom}
Let $s\in\mathbb{R}$ 
and $u\in C(\mathbb{R}^N\setminus\{0\})$ 
be homogeneous of degree $s$, 
that is, 
$$
u(\lambda x)
=
\lambda^s 
u(x)
\text{\;\;for\/\;}
x\in\mathbb{R}^N\setminus\{0\},\,
\lambda>0.
$$
\begin{itemize}
\item[\rm{(i)}]
If $v\in C(\mathbb{R}^N\setminus\{0\})$ 
is homogeneous of degree $s$ as well, 
then so is $u+v$. 
\item[\rm{(ii)}]
For $\nu\in\mathbb{R}$, 
$\lvert u\rvert^\nu$ is homogeneous of degree $s\nu$.
\item[\rm{(iii)}]
If $u\in C^1(\mathbb{R}^N\setminus\{0\})$ 
and $i\in I_N$, 
then $D_i u$ is homogeneous of degree $s-1$.
\end{itemize}
\end{prop}
\par
For a square matrix $A$ of order $N$, 
let us define 
$$
A[x]
=
{}^{\rm t}(A\,{}^{\rm t}\!x)
=
x\,{}^{\rm t}\!A
\text{\;\;for\;}
x=(x_1,x_2,\ldots,x_N)\in\mathbb{R}^N.
$$
\par
\begin{prop}
\label{prop:rad_A}
Let $s\in\mathbb{R}$ 
and $u\in C^1(\mathbb{R}^N\setminus\{0\})$ 
be homogeneous of degree $s$ and radially symmetric, 
that is, 
$$
u(A[x])
=
u(x)
\text{\;\;for\/\;}
x\in\mathbb{R}^N\setminus\{0\},\,
A\in O(N),
$$
where $O(N)$ denotes the orthogonal group of order $N$. 
Then there exists a constant $c\in\mathbb{R}$ 
such that 
$$
u(x)
=
c\lvert x\rvert_N^s
\text{\;\;for\/\;}
x\in\mathbb{R}^N\setminus\{0\}.
$$
\end{prop}
\par
To prove Theorem \ref{thm:derivative_log_abs}, 
we need to use the Fourier transform on $\mathbb{R}^N$. 
Let $\mathcal{S}(\mathbb{R}^N)$ 
denote the Schwartz class on $\mathbb{R}^N$. 
Define the Fourier transform $\mathcal{F}_N$ 
and its inverse $\mathcal{F}_N^{-1}$ on $\mathbb{R}^N$ by
\begin{align*}
\mathcal{F}_Nu(\xi)
&=
\frac{1}
     {(2\pi)^{N/2}}
\int_{\mathbb{R}^N}
e^{-\sqrt{-1}\,(x,\xi)_N}
u(x)
dx
\text{\;\;for\;}
\xi\in\mathbb{R}^N,
\\
\mathcal{F}_N^{-1}u(x)
&=
\frac{1}
     {(2\pi)^{N/2}}
\int_{\mathbb{R}^N}
e^{\sqrt{-1}\,(x,\xi)_N}
u(\xi)
d\xi
\text{\;\;for\;}
x\in\mathbb{R}^N,\,
u\in\mathcal{S}(\mathbb{R}^N),
\end{align*}
respectively, 
where 
$\sqrt{-1}$ denotes the imaginary unit 
and 
$$
(x,\xi)_N
=
\sum_{i=1}^N
x_i \xi_i
\text{\;\;for\;}
x=(x_1,x_2,\ldots,x_N),\xi=(\xi_1,\xi_2,\ldots,\xi_N)
\in\mathbb{R}^N.
$$
%
\par
The crux of Theorem \ref{thm:derivative_log_abs}
is the following observation by using the Fourier transform. 
\par
\begin{lem}
\label{lem:rad}
If $u\in\mathcal{S}(\mathbb{R}^N)$ 
is real-valued and radially symmetric, 
then so is $\lvert\nabla_{\!N}^k u\rvert_{N^k}^2$ 
for $k\in\mathbb{Z}_+$.
\end{lem}
\par
\begin{proof}
The conclusion is trivial if $k=0$; 
we may assume $k\in\mathbb{N}$ below.  
Let $i=(i_1,i_2,\ldots,i_k)\in I_N^k$. 
By the Fourier inversion formula and \cite[Proposition 2.2.11 (10)]{Grafakos1}, 
we have two expressions of $D_i u(x)$; 
\begin{align*}
D_i u(x)
&=
D_{i_1}
D_{i_2}
\cdots
D_{i_k}
[
\mathcal{F}_N^{-1}
[
\mathcal{F}_N
u
]
]
(x)
\\
&=
\frac{(\sqrt{-1})^k}
     {(2\pi)^{N/2}}
\int_{\mathbb{R}^N}
e^{\sqrt{-1}\,(x,\xi)_N}
\xi_{i_1}
\xi_{i_2}
\cdots
\xi_{i_k}
\mathcal{F}_N
u(\xi)
d\xi
\text{\;\;for\;}
x\in\mathbb{R}^N
\end{align*}
and
\begin{align*}
D_i u(x)
&=
D_{i_1}
D_{i_2}
\cdots
D_{i_k}
[
\mathcal{F}_N
[
\mathcal{F}_N^{-1}
u
]
]
(x)
\\
&=
\frac{(-\sqrt{-1})^k}
     {(2\pi)^{N/2}}
\int_{\mathbb{R}^N}
e^{-\sqrt{-1}\,(x,\eta)_N}
\eta_{i_1}
\eta_{i_2}
\cdots
\eta_{i_k}
\mathcal{F}_N^{-1}
u(\eta)
d\eta
\text{\;\;for\;}
x\in\mathbb{R}^N.
\end{align*}
Thus we deduce 
\begin{gather*}
\begin{aligned}
{}&
(D_i u(x))^2
\\
&=
\frac{1}
     {(2\pi)^N}
\iint_{\mathbb{R}^N \times \mathbb{R}^N}
e^{\sqrt{-1}\,(x,\xi-\eta)_N}
\xi_{i_1}
\xi_{i_2}
\cdots
\xi_{i_k}
\eta_{i_1}
\eta_{i_2}
\cdots
\eta_{i_k}
\mathcal{F}_N
u(\xi)
\mathcal{F}_N^{-1}
u(\eta)
d\xi\,
d\eta
\end{aligned}
\\
\text{for\;}
x\in\mathbb{R}^N.
\end{gather*}
Hence we obtain 
\begin{align*}
{}&\lvert\nabla_{\!N}^k u(x)\rvert_{N^k}^2
\\
&=
\sum_{i\in I_N^k}
(D_i u(x))^2
\\
&=
\frac{1}
     {(2\pi)^N}
\sum_{i\in I_N^k}
\iint_{\mathbb{R}^N \times \mathbb{R}^N}
e^{\sqrt{-1}\,(x,\xi-\eta)_N}
\xi_{i_1}
\xi_{i_2}
\cdots
\xi_{i_k}
\eta_{i_1}
\eta_{i_2}
\cdots
\eta_{i_k}
\\*[-5pt]
&\qquad\qquad\qquad\qquad\qquad
\times{}
\mathcal{F}_N
u(\xi)
\mathcal{F}_N^{-1}
u(\eta)
d\xi\,
d\eta
\\
&=
\frac{1}
     {(2\pi)^N}
\iint_{\mathbb{R}^N \times \mathbb{R}^N}
e^{\sqrt{-1}\,(x,\xi-\eta)_N}
(\xi,\eta)_N^k
\mathcal{F}_N
u(\xi)
\mathcal{F}_N^{-1}
u(\eta)
d\xi\,
d\eta
\text{\;\;for\;}
x\in\mathbb{R}^N.
\end{align*}
For $A\in O(N)$, 
we have 
$$
(A[x],y)_N
=
(x,{}^{\rm t}\!A[y])_N
,\;
(A[x],A[y])_N
=
(x,y)_N
\text{\;\;for\;}
x,y\in\mathbb{R}^N.
$$
Since Fourier transform and its inverse 
of a radially symmetric function 
are also radially symmetric 
(see e.g. \cite[Proposition 2.2.11 (13)]{Grafakos1}), 
we see that 
$$
[\mathcal{F}_N u](A[\xi])
=
\mathcal{F}_N u(\xi)
,\;
[\mathcal{F}_N^{-1} u](A[\xi])
=
\mathcal{F}_N^{-1} u(\xi)
\text{\;\;for\;}
\xi\in\mathbb{R}^N.
$$
Changing variables 
$(\xi,\eta)=(A[\tilde\xi],A[\tilde\eta])$, 
we have
\begin{align*}
{}&
\lvert\nabla_{\!N}^k u(A[x])\rvert_{N^k}^2
\\
&=
\frac{1}
     {(2\pi)^N}
\iint_{\mathbb{R}^N \times \mathbb{R}^N}
e^{\sqrt{-1}\,(A[x],\xi-\eta)_N}
(\xi,\eta)_N^k
\mathcal{F}_N
u(\xi)
\mathcal{F}_N^{-1}
u(\eta)
d\xi\,
d\eta
\\
&=
\frac{1}
     {(2\pi)^N}
\iint_{\mathbb{R}^N \times \mathbb{R}^N}
e^{\sqrt{-1}\,(x,{}^{\rm t}\!A[\xi-\eta])_N}
(\xi,\eta)_N^k
\mathcal{F}_N
u(\xi)
\mathcal{F}_N^{-1}
u(\eta)
d\xi\,
d\eta
\\
&=
\frac{1}
     {(2\pi)^N}
\iint_{\mathbb{R}^N \times \mathbb{R}^N}
e^{\sqrt{-1}\,(x,\tilde\xi-\tilde\eta)_N}
(A[\tilde\xi],A[\tilde\eta])_N^k
[\mathcal{F}_N
u](A[\tilde\xi])
[\mathcal{F}_N^{-1}
u](A[\tilde\eta])
d\tilde\xi\,
d\tilde\eta
\\
&=
\frac{1}
     {(2\pi)^N}
\iint_{\mathbb{R}^N \times \mathbb{R}^N}
e^{\sqrt{-1}\,(x,\tilde\xi-\tilde\eta)_N}
(\tilde\xi,\tilde\eta)_N^k
\mathcal{F}_N
u(\tilde\xi)
\mathcal{F}_N^{-1}
u(\tilde\eta)
d\tilde\xi\,
d\tilde\eta
\\
&=
\lvert\nabla_{\!N}^k u(x)\rvert_{N^k}^2,
\end{align*}
which shows that 
$\lvert\nabla_{\!N}^k u\rvert_{N^k}^2$ is radially symmetric. 
\end{proof}
\par
We conclude the proof of Theorem \ref{thm:derivative_log_abs}. 
Let 
$$
B_r^N=\{x\in\mathbb{R}^N;\,\lvert x\rvert_N<r\}
\text{\;\;for\;}
r>0. 
$$
\par
\begin{proof}[Proof of Theorem \ref{thm:derivative_log_abs}]
For $j\in\mathbb{N}$, 
choose $\psi_j\in C_{\rm c}^\infty((0,\infty))$ satisfying
$$
\psi_j(r)
=
\begin{cases}
1
&
\text{for\;}
\dfrac1j<r<j,
\smallskip
\\
0
&
\text{for\;}
0<r<\dfrac{1}{2j} \text{\;\;or\;} r>2j.
\end{cases}
$$
Then the functions $\psi_j(\lvert x\rvert_N)
\lvert x\rvert_N^s$,
 $\psi_j(\lvert x\rvert_N)
\log\lvert x\rvert_N$ 
belong to $\mathcal{S}(\mathbb{R}^N)$ and 
are real-valued, 
radially symmetric.
Also, they satisfy
$$
\psi_j(\lvert x\rvert_N)
\lvert x\rvert_N^s
=
\lvert x\rvert_N^s
,\;
\psi_j(\lvert x\rvert_N)
\log\lvert x\rvert_N
=
\log\lvert x\rvert_N
\text{\;\;for\;}
x\in B_j^N\setminus\overline{B_{1/j}^N}.
$$
Since Lemma \ref{lem:rad} yields that 
$\lvert\nabla_{\!N}^k [\psi_j(\lvert x\rvert_N)
\lvert x\rvert_N^s]\rvert_{N^k}^2$ 
and $\lvert\nabla_{\!N}^k [\psi_j(\lvert x\rvert_N)
\log\lvert x\rvert_N]\rvert_{N^k}^2$ 
are radially symmetric, 
we deduce that so are 
$\lvert\nabla_{\!N}^k [\lvert x\rvert_N^s]\rvert_{N^k}^2$ 
and $\lvert\nabla_{\!N}^k [\log\lvert x\rvert_N]\rvert_{N^k}^2$ 
on $B_j^N\setminus\overline{B_{1/j}^N}$, 
and then on $\mathbb{R}^N\setminus\{0\}$ 
because $j\in\mathbb{N}$ is arbitrary. 
\par
(i)
It follows from Proposition \ref{prop:hom} 
that for $i\in I_N^k$, 
the functions $\lvert x\rvert_N^s$, 
$D_i[\lvert x\rvert_N^s]$ 
and $(D_i[\lvert x\rvert_N^s])^2$ 
are homogeneous of degree $s$, 
$s-k$ and $2(s-k)$, 
respectively. 
Hence $\lvert\nabla_{\!N}^k [\lvert x\rvert_N^s]\rvert_{N^k}^2$ 
is also homogeneous of degree $2(s-k)$. 
Then the desired conclusion immediately follows 
from Proposition \ref{prop:rad_A}. 
\par
(ii)
Since
$$
D_i[\log\lvert x\rvert_N]
=
\frac{x_i}
     {\lvert x\rvert_N^2}
\text{\;\;for\;}
x\in\mathbb{R}^N\setminus\{0\},\,
i\in I_N,
$$
we deduce that this function 
is homogeneous of degree $-1$. 
The rest of the proof is quite similar to (i). 
\end{proof}
\par
\section{Proof of Theorem \ref{thm:ind}}
\label{sect:prf_2}
\par
In this section, 
we prove Theorem \ref{thm:ind}. 
We decompose it into the following three lemmas.
\par
\begin{lem}
\label{lem:1-dim}
Theorem \ref{thm:ind} holds true for $N=1$. 
Namely: 
\begin{itemize}
\item[{\rm (i)}]
For any $k\in\mathbb{Z}_+$ 
and $s\in\mathbb{R}$,  
it holds
\begin{align*}
\gamma_1^{s,k}
&=
\left(
k!
\sum_{n=\lceil k/2\rceil}^k
2^{2n-k}
{s/2 \choose n}
{n \choose k-n}
\right)^2.
\end{align*}
\item[{\rm (ii)}]
For any $k\in\mathbb{N}$, 
it holds
\begin{align*}
\ell_1^k
&=
\left(
k!
\sum_{n=\lceil k/2\rceil}^k
2^{2n-k}
\frac{(-1)^n}
     {2n}
{n \choose k-n}
\right)^2.
\end{align*}
\end{itemize}
\end{lem}
\par
\begin{lem}
\label{lem:ind_N}
Let $N\in\mathbb{N}\setminus\{1\}$.
\begin{itemize}
\item[\rm (i)]
For $k\in\mathbb{Z}_+$ and $s\in\mathbb{R}$, 
it holds
$$
\gamma_N^{s,k}
=
k!
\sum_{l=0}^{\lfloor k/2\rfloor}
\frac{(k-2l)!}
     {(2l)!}
\left(
\sum_{n=\lceil k/2\rceil}^{k-l}
2^{2n-k}
{s/2 \choose n}
{n \choose k-n}
{k-n\choose l}
\right)^2
\gamma_{N-1}^{2l,2l}.
$$
In particular, 
for $m\in\mathbb{Z}_+$, 
it holds
$$
\gamma_N^{2m,2m}
=
(2m)!
\sum_{l=0}^m
\frac{(2(m-l))!}
     {(2l)!}
{m \choose l}^2
\gamma_{N-1}^{2l,2l}.
$$
\item[\rm (ii)]
For $k\in\mathbb{N}$, 
it holds
$$
\ell_N^k
=
k!
\sum_{l=0}^{\lfloor k/2\rfloor}
\frac{(k-2l)!}
     {(2l)!}
\left(
\sum_{n=\lceil k/2\rceil}^{k-l}
2^{2n-k}
\frac{(-1)^n}
     {2n}
{n \choose k-n}
{k-n\choose l}
\right)^2
\gamma_{N-1}^{2l,2l}.
$$
\end{itemize}
\end{lem}
\par
\begin{lem}
\label{lem:2m}
For $N\in\mathbb{N}$ and $m\in\mathbb{Z}_+$, 
it holds
\begin{equation}
\label{eq:2m}
\gamma_N^{2m,2m}
=
2^{2m}
m!\,
(2m)!\,
\left(
\frac N2+m-1
\right)_m.
\end{equation}
\end{lem}
\par
Combining these three lemmas yields Theorem \ref{thm:ind}. 
We now concentrate on proving them. 
We need some propositions. 
For $m\in\mathbb{Z}_+$, 
define 
$$
\phi^m(t)=(t^2+2t)^m 
\text{\;\;for\;}
t\in\mathbb{R}.
$$ 
\par
\begin{prop}
\label{prop:Kronecker}
Let $m,k\in\mathbb{Z}_+$.
\begin{itemize}
\item[\rm (i)]
It holds 
$$
[\phi^m]^{(k)}
(0)
=
\chi^{}_{[m,2m]}(k)
2^{2m-k}k!\,
{m\choose k-m},
$$
where\/ $\chi_S$ denotes 
the characteristic function of a set $S$. 
\item[\rm (ii)]
It holds
$$
\lvert
\nabla_{\!N}^k
[\lvert\,\cdot\,\rvert_N^{2m}]
(0)
\rvert_{N^k}^2
=
\delta_{k,2m}
\gamma_N^{2m,2m}
=
\delta_{k,2m}
\gamma_N^{k,k}.
$$
\end{itemize}
\end{prop}
\par
\begin{proof}
(i) 
Expand $\phi^m$ by means of the binomial theorem;
$$
\phi^m(t)
=
\sum_{j=0}^m
2^{m-j}
{m\choose j}
t^{m+j}
\text{\;\;for\;}
t\in\mathbb{R}.
$$
Let $\nu_+=\max\{\nu,0\}$. 
For $k\in\mathbb{Z}_+$, 
we have 
\begin{align*}
[\phi^m]^{(k)}(t)
&=
\sum_{j=(k-m)_+}^m
2^{m-j}
{m\choose j}
(m+j)_k
t^{m+j-k}
\\
&=
\sum_{l=(m-k)_+}^{2m-k}
2^{2m-k-l}
{m\choose k-m+l}
(k+l)_k
t^l
\text{\;\;for\;}
t\in\mathbb{R},
\end{align*}
which implies the assertion. 
\par
(ii) 
If $k>2m$ and $i\in I_N^k$, 
then 
$D_i[\lvert x\rvert_N^{2m}]=0$ 
for $x\in\mathbb{R}^N$, 
which implies 
$$
\lvert
\nabla_{\!N}^k
[\lvert x\rvert_N^{2m}]
\rvert_{N^k}^2
=
0
\text{\;\;for\;}
x\in\mathbb{R}^N. 
$$
Meanwhile, if $k\le 2m$, 
then Theorem \ref{thm:derivative_log_abs} (i) 
shows that 
$$
\lvert
\nabla_{\!N}^k
[\lvert x\rvert_N^{2m}]
\rvert_{N^k}^2
=
\gamma_N^{2m,k}
\lvert x\rvert_N^{2(2m-k)}
\text{\;\;for\;}
x\in\mathbb{R}^N\setminus\{0\}.
$$
Hence a passage to the limit as  
$x\to0$ 
yields the assertion. 
\end{proof}
\par
In what follows, 
we use the notation 
$$
x'=(x_1,x_2,\ldots,x_{N-1})
\in\mathbb{R}^{N-1}
\text{\;\;for\;}
x=(x_1,x_2,\ldots,x_{N-1},x_N)\in\mathbb{R}^N 
$$
when $N\in\mathbb{N}\setminus\{1\}$. 
Let $\Omega$ be a domain in $\mathbb{R}^N$, 
and for $u\in C^k(\Omega)$, 
we write 
$$
\lvert
\nabla_{\!N-1}^k u(x)
\rvert_{(N-1)^k}^2
=
\sum_{i'\in I_{N-1}^k}
(D_{i'} u(x))^2
\text{\;\;for\;}
x\in\Omega.
$$
\par
\begin{prop}
\label{prop:N_decomposite}
Let $N\in\mathbb{N}\setminus\{1\}$, 
$k\in\mathbb{Z}_+$ 
and $\Omega$ be a domain in\/ $\mathbb{R}^N$. 
Then for $u\in C^k(\Omega)$, 
we have
$$
\lvert
\nabla_{\!N}^k u(x)
\rvert_{N^k}^2
=
\sum_{j=0}^k
{k\choose j}
\lvert
\nabla_{\!N-1}^j 
[D_N^{k-j}u](x)
\rvert_{(N-1)^j}^2
\text{\;\;for\/\;}
x\in\Omega.
$$
\end{prop}
\par
\begin{proof}
The conclusion is trivial if $k=0$; 
we may assume $k\in\mathbb{N}$ below.  
Define
$$
j_{N-1}[i]
=
\sharp\{n\in\{1,2,\ldots,k\};\,
i_n\in I_{N-1}\}
\text{\;\;for\;}
i=(i_1,i_2,\ldots,i_k)\in I_N^k
$$
and
$$
I_{N-1}^{j;k}
=
\{i\in I_N^k;\,
j_{N-1}[i]=j\}
\text{\;\;for\;}
j\in\{0,1,\ldots,k\}.
$$
For $i=(i_1,i_2,\ldots,i_k)\in I_{N-1}^{j;k}$, 
let 
$$
(n'_1[i],n'_2[i],\ldots,n'_j[i])
$$
be all the $n$'s listed in ascending order 
such that $i_n\in I_{N-1}$, 
and let
$$
(\tilde n_1[i],\tilde n_2[i],\ldots,\tilde n_{k-j}[i])
$$
be all the $n$'s listed in ascending order 
such that $i_n=N$. 
If we define
$$
i'_{N-1}[i]
=
(i_{n'_1[i]},i_{n'_2[i]},\ldots,i_{n'_j[i]})
,\;
\tilde i_N[i]
=
(i_{\tilde n_1[i]},i_{\tilde n_2[i]},\ldots,i_{\tilde n_{k-j}[i]}),
$$
then
$$
i'_{N-1}[i]\in I_{N-1}^j
,\;
\tilde i_N[i]=(N,N,\ldots,N)\in\{N\}^{k-j}
\text{\;\;for\;}
i\in I_{N-1}^{j;k}
$$
and
$$
D_iu(x)
=
D_{i'_{N-1}[i]}
D_{\tilde i_N[i]}u(x)
=
D_{i'_{N-1}[i]}
[D_N^{k-j}u](x)
\text{\;\;for\;}
x\in\Omega,\,
i\in I_{N-1}^{j;k}.
$$
We next define
\begin{gather*}
\Sigma_k^{k-j}
=
\{
\sigma=(\sigma_1,\sigma_2,\ldots,\sigma_{k-j})
\in\{1,2,\ldots,k\}^{k-j};\,
\sigma_1<\sigma_2<\cdots<\sigma_{k-j}\}
\\
\text{for\;}
j\in\{0,1,\ldots,k-1\}.
\end{gather*}
For $\sigma\in\Sigma_k^{k-j}$, 
define
$$
I_{N-1}^{j;k}
(\sigma)
=\{i\in I_{N-1}^{j;k}
;\,
(\tilde n_1[i],\tilde n_2[i],\ldots,\tilde n_j[i])
=\sigma\}.
$$
Since the mapping 
$
I_{N-1}^{j;k}(\sigma)
\ni
i\mapsto i'_{N-1}[i]
\in I_{N-1}^j
$ 
is bijective, 
we have
\begin{align*}
\sum_{i\in I_{N-1}^{j;k}(\sigma)}
(D_i u(x))^2
&=
\sum_{i\in I_{N-1}^{j;k}(\sigma)}
(D_{i'_{N-1}[i]}
[D_N^{k-j}u](x)
)^2
\\
&=
\sum_{i'\in I_{N-1}^j}
(D_{i'}
[D_N^{k-j}u](x)
)^2
\\
&=
\lvert
\nabla_{\!N-1}^j
[D_N^{k-j}u](x)
\rvert_{(N-1)^j}^2
\text{\;\;for\;}
x\in\Omega.
\end{align*}
Since 
$$
\sharp\Sigma_k^{k-j}
=
{k\choose k-j}
=
{k\choose j}
\text{\;\;for\;}
j\in\{0,1,\ldots,k-1\},
$$
we deduce
\begin{align*}
\lvert
\nabla_{\!N}^k
u(x)
\rvert_{N^k}^2
&=
\sum_{i\in I_N^k\setminus I_{N-1}^k}
(D_iu(x))^2
+
\sum_{i'\in I_{N-1}^k}
(D_{i'}u(x))^2
\\
&=
\sum_{j=0}^{k-1}
\sum_{\sigma\in\Sigma_k^{k-j}}
\sum_{i\in I_{N-1}^{j;k}(\sigma)}
(D_iu(x))^2
+
\lvert
\nabla_{\!N-1}^k
u(x)
\rvert_{(N-1)^k}^2
\\
&=
\sum_{j=0}^k
{k\choose j}
\lvert
\nabla_{\!N-1}^j
[D_N^{k-j}u](x)
\rvert_{(N-1)^j}^2
\text{\;\;for\;}
x\in\Omega.
\end{align*}
This completes the proof.
\end{proof}
\par
Define $e_N=(0,0,\ldots,0,1)\in\mathbb{R}^N$ and
$$
\rho_N(x)
=
\lvert x+e_N\rvert_N^2-1
\text{\;\;for\;}
x\in\mathbb{R}^N, 
$$
which becomes
$$
\rho_N(x)
=
\begin{cases}
\phi^1(x)
&
\text{for\;}
x\in\mathbb{R}\phantom{^N}
\text{\;\;if\;}
N=1,
\smallskip
\\
\lvert x'\rvert_{N-1}^2
+\phi^1(x_N)
&
\text{for\;}
x\in\mathbb{R}^N
\text{\;\;if\;}
N\in\mathbb{N}\setminus\{1\}.
\end{cases}
$$
Note that 
$$
\lvert \rho_N(x)\rvert
=
\lvert\,\lvert x\rvert_N^2+2x_N\rvert
\le
\lvert x\rvert_N^2+2\lvert x_N\rvert
\le
\lvert x\rvert_N^2+2\lvert x\rvert_N
<
\varepsilon
\text{\;\;for\;}
x\in B_{(1+\varepsilon)^{1/2}-1}^N
$$
for all $\varepsilon>0$. 
\par
\begin{prop}
\label{prop:Taylor}
Let $\varepsilon>0$ and
$$
f(t)
=
\sum_{n=0}^\infty
a_n t^n
\text{\;\;for\/\;}
-\varepsilon<t<\varepsilon
$$
be analytic, 
where\/ $\{a_n\}_{n=0}^\infty\subset\mathbb{R}$. 
\begin{itemize}
\item[{\rm (i)}]
Let $N=1$  
and $k\in\mathbb{Z}_+$. 
Then it holds 
$$
[f(\rho_1)]^{(k)}(0)
=
k!
\sum_{n=\lceil k/2\rceil}^k
2^{2n-k}
a_n
{n \choose k-n}.
$$
\item[{\rm (ii)}]
Let $N\in\mathbb{N}\setminus\{1\}$ 
and $k\in\mathbb{Z}_+$. 
Then it holds 
\begin{align*}
{}&
\lvert
\nabla_{\!N}^k
[f(\rho_N)](0)
\rvert_{N^k}^2
\\
&=
k!
\sum_{l=0}^{\lfloor k/2\rfloor}
\frac{(k-2l)!}
     {(2l)!}
\left(
\sum_{n=\lceil k/2\rceil}^{k-l}
2^{2n-k}
a_n
{n \choose k-n}
{k-n\choose l}
\right)^2
\gamma_{N-1}^{2l,2l}.
\end{align*}
\end{itemize}
\end{prop}
\par
\begin{proof}
(i) 
It follows from the definition of $\rho_1$ and $\phi^n$ that 
\begin{gather*}
[f(\rho_1)]^{(k)}(x)
=
[f(\phi^1)]^{(k)}(x)
=
\sum_{n=0}^\infty
a_n[\phi^n]^{(k)}(x)
\\
\text{for\;}
-((1+\varepsilon)^{1/2}-1)<x<(1+\varepsilon)^{1/2}-1.
\end{gather*}
If we invoke Proposition \ref{prop:Kronecker} (i),
then we have 
$$
[f(\rho_1)]^{(k)}(0)
=
\sum_{n=0}^\infty
a_n[\phi^n]^{(k)}(0)
=
k!
\sum_{n=\lceil k/2\rceil}^k
2^{2n-k}
a_n
{n \choose k-n}.
$$
Thus, (i) is established.
\par
(ii) 
Using the binomial expansion, 
we have
\begin{align*}
f(\rho_N(x))
&=
\sum_{n=0}^\infty
a_n (\rho_N(x))^n
\\
&=
\sum_{n=0}^\infty
a_n (\lvert x'\rvert_{N-1}^2
+\phi^1(x_N)
)^n
\\
&=
\sum_{n=0}^\infty
a_n
\sum_{m=0}^n
{n\choose m}
\phi^{n-m}(x_N)
 \lvert x'\rvert_{N-1}^{2m}
\text{\;\;for\;}
x\in B_{(1+\varepsilon)^{1/2}-1}^N.
\end{align*}
Proposition \ref{prop:N_decomposite} gives 
\begin{gather*}
\begin{aligned}
{}&
\lvert
\nabla_{\!N}^k [f(\rho_N)](x)
\rvert_{N^k}^2
\\
&=
\sum_{j=0}^k
{k\choose j}
\lvert
\nabla_{\!N-1}^j 
[D_N^{k-j}[f(\rho_N)]](x)
\rvert_{(N-1)^j}^2
\\
&=
\sum_{j=0}^k
{k\choose j}
\left\lvert
\sum_{n=0}^\infty
a_n
\sum_{m=0}^n
{n\choose m}
[\phi^{n-m}]^{(k-j)}(x_N)
\nabla_{\!N-1}^j 
[\lvert x'\rvert_{N-1}^{2m}]
\right\rvert_{(N-1)^j}^2
\\
&=
\sum_{j=0}^{\lfloor k/2\rfloor}
{k\choose 2j+1}
\left\lvert
\sum_{n=0}^\infty
a_n
\sum_{m=0}^n
{n\choose m}
[\phi^{n-m}]^{(k-2j-1)}(x_N)
\nabla_{\!N-1}^{2j+1} 
[\lvert x'\rvert_{N-1}^{2m}]
\right\rvert_{(N-1)^{2j+1}}^2
\\*
&\quad{}+
\sum_{j=0}^{\lfloor k/2\rfloor}
{k\choose 2j}
\left\lvert
\sum_{n=0}^\infty
a_n
\sum_{m=0}^n
{n\choose m}
[\phi^{n-m}]^{(k-2j)}(x_N)
\nabla_{\!N-1}^{2j} 
[\lvert x'\rvert_{N-1}^{2m}]
\right\rvert_{(N-1)^{2j}}^2
\end{aligned}
\\
\text{for\;}
x\in B_{(1+\varepsilon)^{1/2}-1}^n.
\end{gather*}
Here, 
we decomposed the summation with respect to $j$ 
into two parts 
consisting odd $j$'s and even $j$'s. 
Note that Proposition \ref{prop:Kronecker} (ii) 
gives $\nabla_{\!N-1}^j 
[\lvert\,\cdot\,\rvert_{N-1}^{2m}](0)=0$ unless $j=2m$. 
It follows from Proposition \ref{prop:Kronecker} (i) 
that
$$
{n\choose l}
[\phi^{n-l}]^{(k-2l)}(0)
=
\chi^{}_{[k/2,k-l]}(n)
2^{2n-k}
(k-2l)!\,
{n \choose k-n}
{k-n\choose l}.
$$
Using these equalities, 
we have
\begin{align*}
{}&
\lvert
\nabla_{\!N}^k [f(\rho_N)](0)
\rvert_{N^k}^2
\\
&=
\sum_{l=0}^{\lfloor k/2\rfloor}
{k\choose 2l+1}
\left\lvert
\sum_{n=0}^\infty
a_n
\sum_{m=0}^n
{n\choose m}
[\phi^{n-m}]^{(k-2l-1)}(0)
\nabla_{\!N-1}^{2l+1} 
[\lvert\,\cdot\,\rvert_{N-1}^{2m}](0)
\right\rvert_{(N-1)^{2l+1}}^2
\\*
&\quad{}
+
\sum_{l=0}^{\lfloor k/2\rfloor}
{k\choose 2l}
\left\lvert
\sum_{n=0}^\infty
a_n
\sum_{m=0}^n
{n\choose m}
[\phi^{n-m}]^{(k-2l)}(0)
\nabla_{\!N-1}^{2l} 
[\lvert\,\cdot\,\rvert_{N-1}^{2m}](0)
\right\rvert_{(N-1)^{2l}}^2
\\
&=
\sum_{l=0}^{\lfloor k/2\rfloor}
{k\choose 2l}
\left\lvert
\sum_{n=l}^\infty
a_n
{n\choose l}
[\phi^{n-l}]^{(k-2l)}(0)
\nabla_{\!N-1}^{2l} 
[\lvert\,\cdot\,\rvert_{N-1}^{2l}](0)
\right\rvert_{(N-1)^{2l}}^2
\\
&=
k!
\sum_{l=0}^{\lfloor k/2\rfloor}
\frac{(k-2l)!}
     {(2l)!}
\left(
\sum_{n=\lceil k/2\rceil}^{k-l}
2^{2n-k}
a_n
{n \choose k-n}
{k-n\choose l}
\right)^2
\gamma_{N-1}^{2l,2l}.
\end{align*}
\end{proof}
\par
For $s\in\mathbb{R}$, 
define 
$$
f_s(t)
=
(1+t)^{s/2}
,\;
f_\ast(t)
=
\frac12
\log(1+t)
\text{\;\;for\;}
-1<t<1.
$$
Then the Taylor expansion formula 
(see e.g. \cite[p.\ 361]{Bers1})
immediately yields 
$$
f_s(t)
=
\sum_{n=0}^\infty
{s/2 \choose n}
t^n
,\;
f_\ast(t)
=
\sum_{n=1}^\infty
\frac{(-1)^{n-1}}
     {2n}
t^n
\text{\;\;for\;}
-1<t<1.
$$
We now prove Lemmas \ref{lem:1-dim} and \ref{lem:ind_N} 
by applying Proposition \ref{prop:Taylor}. 
First we prove Lemma \ref{lem:ind_N}. 
\par
\begin{proof}[Proof of Lemma \ref{lem:ind_N}]
Since $\lvert e_N\rvert_N=1$ and  
$$
\lvert x+e_N\rvert_N^s
=
f_s(\rho_N(x))
,\;
\log\lvert x+e_N\rvert_N
=
f_\ast(\rho_N(x))
\text{\;\;for\;}
x\in B_{2^{1/2}-1}^N,
$$
we deduce 
\begin{gather*}
\gamma_N^{s,k}
=
\lvert
\nabla_{\!N}^k
[\lvert\,\cdot\,\rvert_N^s]
(e_N)
\rvert_{N^k}^2
=
\lvert
\nabla_{\!N}^k
[\lvert\,\cdot+e_N\rvert_N^s]
(0)
\rvert_{N^k}^2
=
\lvert
\nabla_{\!N}^k
[f_s(\rho_N)]
(0)
\rvert_{N^k}^2,
\\
\ell_N^k
=
\lvert
\nabla_{\!N}^k
[\log\lvert\,\cdot\,\rvert_N]
(e_N)
\rvert_{N^k}^2
=
\lvert
\nabla_{\!N}^k
[\log\lvert\,\cdot+e_N\rvert_N]
(0)
\rvert_{N^k}^2
=
\lvert
\nabla_{\!N}^k
[
f_\ast(\rho_N)
]
(0)
\rvert_{N^k}^2.
\end{gather*}
Applying Proposition \ref{prop:Taylor} (ii), 
we obtain both the assertions (i) and (ii). 
\end{proof}
\par
Next we prove Lemma \ref{lem:1-dim}. 
\par
\begin{proof}[Proof of Lemma \ref{lem:1-dim}]
We argue as in the proof of Lemma \ref{lem:ind_N} 
with applying Proposition \ref{prop:Taylor} (i) instead of Proposition \ref{prop:Taylor} (ii) 
to obtain the assertion. 
\end{proof}
\par
We need the following proposition 
to prove Lemma \ref{lem:2m}.
\par
\begin{prop}
\label{prop:m+1/2}
For $\nu\in\mathbb{R}$ and $m\in\mathbb{Z}_+$, 
it holds 
\begin{equation}
\label{eq:ind_l}
\sum_{l=0}^m
\frac{(2l)!}
     {2^{2l}l!}
(\nu+m-l)_{m-l}
{m\choose l}
=
\left(
\nu+m+\frac12
\right)_m.
\end{equation}
\end{prop}
\par
\begin{proof}
We use an induction on $m$. 
When $m=0$, 
(\ref{eq:ind_l}) trivially holds. 
Fix $m\in\mathbb{N}$ and assume 
that (\ref{eq:ind_l}) holds for $m-1$, 
that is, 
\begin{equation}
\label{eq:****}
\sum_{l=0}^{m-1}
\frac{(2l)!}
     {2^{2l}l!}
(\nu+m-l-1)_{m-l-1}
{m-1\choose l}
=
\left(
\nu+m-\frac12
\right)_{m-1}.
\end{equation}
We use the following identities 
\begin{gather*}
{m\choose l}
=
{m-1\choose l}
+
{m-1\choose l-1}
\text{\;\;for\;}
l\in\mathbb{N}
,
\\
(\nu+m-l)_{m-l}
=
(\nu+m-l)
(\nu+m-l-1)_{m-l-1}
\text{\;\;for\;}
l\in\mathbb{Z}_+.
\end{gather*}
Then we have
\begin{align*}
{}&
\sum_{l=0}^m
\frac{(2l)!}
     {2^{2l}l!}
(\nu+m-l)_{m-l}
{m\choose l}
\\
&=
(\nu+m)_m
+
\sum_{l=1}^{m-1}
\frac{(2l)!}
     {2^{2l}l!}
(\nu+m-l)_{m-l}
{m-1\choose l}
\\*
&\quad{}
+
\sum_{l=1}^{m-1}
\frac{(2l)!}
     {2^{2l}l!}
(\nu+m-l)_{m-l}
{m-1\choose l-1}
+
\frac{(2m)!}
     {2^{2m}m!}
\\
&=
\sum_{l=0}^{m-1}
\frac{(2l)!}
     {2^{2l}l!}
(\nu+m-l)
(\nu+m-l-1)_{m-l-1}
{m-1\choose l}
\\*
&\quad{}
+
\sum_{l=0}^{m-1}
\frac{(2(l+1))!}
     {2^{2(l+1)}(l+1)!}
(\nu+m-l-1)_{m-l-1}
{m-1\choose l}
\\
&=
\left(
\nu+m+\frac12
\right)
\sum_{l=0}^{m-1}
\frac{(2l)!}
     {2^{2l}l!}
(\nu+m-l-1)_{m-l-1}
{m-1\choose l}.
\end{align*}
Applying (\ref{eq:****}),
we have
$$
\sum_{l=0}^m
\frac{(2l)!}
     {2^{2l}l!}
(\nu+m-l)_{m-l}
{m\choose l}
=
\left(
\nu+m+\frac12
\right)_m,
$$
which shows that (\ref{eq:ind_l}) holds also for $m$. 
The calculation above works also for $m=1$; 
as usual, 
we regard any empty sum as 0. 
Thus (\ref{eq:ind_l}) is proved. 
\end{proof}
\par
We now prove Lemma \ref{lem:2m}.
\par
\begin{proof}[Proof of Lemma \ref{lem:2m}]
We use an induction on $N$. 
First, 
(\ref{eq:1-dim_result}) gives 
$$
\gamma_1^{2m,2m}=((2m)!)^2
\text{\;\;for\;}
m\in\mathbb{Z}_+.
$$
Meanwhile we have
$$
2^{2m}
m!\,
(2m)!\,
\left(
m-\frac12
\right)_m
=
2^m 
m!\,
(2m)!
\prod_{j=1}^m (2j-1)
=
((2m)!)^2
\text{\;\;for\;}
m\in\mathbb{Z}_+.
$$
The equality above 
is valid also for $m=0$; 
as usual, 
we regard any empty product as 1. 
Thus (\ref{eq:2m}) holds for $N=1$. 
Fix $N\in\mathbb{N}\setminus\{1\}$ and assume 
that (\ref{eq:2m}) holds for $N-1$, 
that is, 
\begin{equation}
\label{eq:++++}
\gamma_{N-1}^{2m,2m}
=
2^{2m}
m!\,
(2m)!\,
\left(
\frac{N-3}
     {2}
+m
\right)_m
\text{\;\;for\;}
m\in\mathbb{Z}_+.
\end{equation}
It follows from Lemma \ref{lem:ind_N} that 
\begin{align*}
\gamma_N^{2m,2m}
&=
(2m)!
\sum_{l=0}^m
\frac{(2(m-l))!}
     {(2l)!}
{m \choose l}^2
\gamma_{N-1}^{2l,2l}
\\
&=
(2m)!
\sum_{l=0}^m
\frac{(2l)!}
     {(2(m-l))!}
{m \choose l}^2
\gamma_{N-1}^{2(m-l),2(m-l)}.
\end{align*}
Applying (\ref{eq:++++}) and Proposition \ref{prop:m+1/2}, 
we have
\begin{align*}
\gamma_N^{2m,2m}
&=
(2m)!
\sum_{l=0}^m
(2l)!\,
{m \choose l}^2
2^{2(m-l)}
(m-l)!\,
\left(
\frac{N-3}
     {2}
+m-l
\right)_{m-l}
\\
&=
2^{2m} 
m!\, 
(2m)!
\sum_{l=0}^m
\frac{(2l)!}
     {2^{2l}l!}
\left(
\frac{N-3}
     {2}
+m-l
\right)_{m-l}
{m\choose l}
\\
&=
2^{2m} 
m!\, 
(2m)!\,
\left(
\frac N2+m-1
\right)_m
\text{\;\;for\;}
m\in\mathbb{Z}_+,
\end{align*}
which shows that (\ref{eq:2m}) holds also for $N$. 
Thus (\ref{eq:2m}) is proved. 
\end{proof}
\par
Thus we have proved Theorem \ref{thm:ind}. 
\section{Proof of Theorem \ref{thm:2-dim}}
\label{sect:prf_3}
\par
We can easily prove 
Theorem \ref{thm:2-dim} 
by applying Theorem \ref{thm:derivative_log_abs}. 
\par
\begin{proof}[Proof of Theorem \ref{thm:2-dim}]
Let $k\in\mathbb{N}$. 
For $u\in C^{k+1}(\mathbb{R}^N\setminus\{0\})$, 
a direct calculation shows  
\begin{equation}
\label{eq:Laplacian1}
\Delta_N
[
\lvert
\nabla_{\!N}^{k-1}
u
\rvert_{N^{k-1}}^2
]
=
2
\lvert
\nabla_{\!N}^k
u
\rvert_{N^k}^2
+
2
(
\nabla_{\!N}^{k-1}
u
,
\nabla_{\!N}^{k-1}
[
\Delta_N
u
]
)_{N^{k-1}}
\text{\;\;on\;}
\mathbb{R}^N\setminus\{0\}, 
\end{equation}
where
$\Delta_N=D_1^2+D_2^2+\cdots+ D_N^2$ 
is the usual Laplacian on $\mathbb{R}^N$. 
We see that for $\nu\in\mathbb{R}$, 
\begin{equation}
\label{eq:Laplacian2}
\Delta_N[\lvert x\rvert_N^\nu]
=
\nu(\nu+N-2)
\lvert x\rvert_N^{\nu-2}
\text{\;\;for\;}
x\in\mathbb{R}^N\setminus\{0\}. 
\end{equation}
\par
(i) 
It follows from (\ref{eq:Laplacian1}) 
and (\ref{eq:Laplacian2}) that 
$$
2
\left\lvert
\nabla_{\!N}^k
\left[
\frac{1}
     {\lvert x\rvert_N^{N-2}}
\right]
\right\rvert_{N^k}^2
=
\Delta_N
\left[
\left\lvert
\nabla_{\!N}^{k-1}
\left[
\frac{1}
     {\lvert x\rvert_N^{N-2}}
\right]
\right\rvert_{N^{k-1}}^2
\right]
\text{\;\;for\;}
x\in\mathbb{R}^N\setminus\{0\}. 
$$
By virtue of Theorem \ref{thm:derivative_log_abs} 
and (\ref{eq:Laplacian2}), 
we deduce 
\begin{align*}
{}&
2\gamma_N^{-(N-2),k}
\frac{1}
     {\lvert x\rvert_N^{2(N+k-2)}}
\\
&=
\gamma_N^{-(N-2),k-1}
\Delta_N
\left[
\frac{1}
     {\lvert x\rvert_N^{2(N+k-3)}}
\right]
\\
&=
2(N+k-3)(N+2k-4)
\gamma_N^{-(N-2),k-1}
\frac{1}
     {\lvert x\rvert_N^{2(N+k-2)}}
\text{\;\;for\;}
x\in\mathbb{R}^N\setminus\{0\}, 
\end{align*}
which implies
$$
\gamma_N^{-(N-2),k}
=
(N+k-3)
\left(
\frac N2+k-2
\right)
\gamma_N^{-(N-2),k-1}. 
$$
The desired conclusion now follows inductively 
since $\gamma_N^{-(n-2),0}=1$. 
\par
(ii) 
Note that
$$
\Delta_2[\log\lvert x\rvert_2]
=
0
\text{\;\;for\;}
x\in\mathbb{R}^2\setminus\{0\}. 
$$
We argue as in (i) to deduce 
$$
\ell_2^k=2(k-1)^2\ell_2^{k-1}.
$$ 
The desired conclusion now follows inductively 
since $\ell_2^1=1$. 
\end{proof}
\par
\section*{Acknowledgment}
The first author is supported by 
the JSPS Global COE Program, 
Kyoto University. 
The third author was supported
by Grant-in-Aid for Young Scientists (B) (No.21740104)
Japan Society for the Promotion of Science.



\end{document}